\documentclass[11pt]{article}
\usepackage[scale=0.8]{geometry}
\usepackage[english]{babel}
\usepackage{amssymb,amsmath,amsthm,bm}
\usepackage{amsfonts}
\usepackage{latexsym}
\usepackage{lineno}
\usepackage{enumitem}
	\setlist[enumerate]{noitemsep,nolistsep}
\usepackage{mathtools}
\usepackage[
pdfauthor={},
pdftitle={},
pdfstartview=XYZ,
bookmarks=true,
colorlinks=true,
linkcolor=blue,
urlcolor=blue,
citecolor=blue,
bookmarks=true,
linktocpage=true,
hyperindex=true
]{hyperref}

\newenvironment{proofclaim}[1][Proof of claim]{\begin{proof}[#1]}{\end{proof}}

\theoremstyle{plain}
\newtheorem{thm}{Theorem}[section]

\newtheorem{lem}[thm]{Lemma}

\newtheorem{conj}[thm]{Conjecture}
\newtheorem{claim}[thm]{Claim}
\newtheorem{fact}[thm]{Fact}
\numberwithin{equation}{section}

\begin{document}
	\title{Equitable tree colouring of graphs}
	
	\author{Yuping Gao\footnote{School of Mathematics and Statistics, Lanzhou University, Lanzhou 730000, China.
			Corresponding author.
			Email: {\tt gaoyp@lzu.edu.cn}.
			Supported in part by  NSFC grant, No. 12271228 and China Scholarship Council, No. 202406180027.}
		\qquad
		Allan Lo\footnote{School of Mathematics, University of Birmingham, Birmingham, B15 2TT, UK. Email: {\tt s.a.lo@bham.ac.uk}.}
		\qquad
		Songling Shan\footnote{Department of Mathematics and Statistics, Auburn University, Auburn, AL 36849, USA.
			Email:	{\tt szs0398@auburn.edu}.
			Supported in part by NSF grant   DMS-2451895.
		}
	}
	
	\date{\today}
	\maketitle
	\begin{abstract}
		Let $k \in \mathbb{N}$ and let $G$ be a simple graph with maximum degree~$\Delta$.
		A $k$-colouring $\varphi$ of $G$ is an assignment of colours from $\{1,2,\ldots,k\}$ to the vertices of $G$.
		We call $\varphi$ \emph{proper} if adjacent vertices receive distinct colours, and  \emph{equitable} if the sizes of any two colour classes differ by at most one.
		The celebrated Hajnal--Szemer\'{e}di  theorem states that a proper equitable $k$-colouring exists whenever $k \ge \Delta + 1$.

		In this paper, we study its tree colouring variant  in which each colour class induces a forest.
		This is closely related to the vertex arboricity which was introduced by Chartrand, Kronk, and Wall.
		More precisely, we prove that if $n \ge 3\Delta^4$ and $k \ge (\Delta+2)/2$, then every $n$-vertex graph with maximum degree at most $\Delta$ contains an equitable tree $k$-colouring.
		This confirms a conjecture of Wu, Zhang, and Li when $\Delta$ is even and up to an additive constant of~$1$ otherwise for large $n$.
		We also consider $d$-degenerate colouring in which each colour class induces a $d$-degenerate graph.
		
\medskip
		
	\end{abstract}

\section{Introduction}\label{sec:introduction}

For integers $p$ and $q$, let  $[p,q] = \{i \in \mathbb{Z} \colon p \le i \le q \}$. In particular, we denote $[1,q]$ by $[q]$. Let $k \in \mathbb{N}$ and let $G$ be a graph. A \emph{$k$-colouring} $\varphi: V(G) \rightarrow [k]$ is an assignment of colours to the vertices of $G$. The set of vertices assigned the same colour under $\varphi$ is a \emph{colour class}.
An old conjecture of Erd\H{o}s~\cite{E1964} states that any graph $G$ has a proper equitable $k$-colouring if $k\ge \Delta(G)+1$. Here, a \emph{proper equitable $k$-colouring} of $G$ is a $k$-colouring of $G$ such that no two adjacent vertices receive the same colour and the sizes of any two distinct colour classes differ by at most one.
This conjecture was proved by Hajnal and Szemer\'{e}di~\cite{HS1970}, and a simpler proof was later provided by Kierstead and Kostochka~\cite{KK2008}.
The \emph{equitable chromatic number} of a graph $G$, denoted by $\chi_{=}(G)$, is the smallest integer $k$ such that $G$ admits a proper equitable $k$-colouring. Note that a graph $G$ which is equitably $k$-colourable may not be equitably $k'$-colourable for some $k'>k$. For example, $\chi_{=}(K_{3,3})=2$, but $K_{3,3}$ has no proper equitable $3$-colouring. The smallest integer $k$ such that $G$ has a proper equitable $k'$-colouring for every number of colours $k'\ge k$ is called the \emph{equitable chromatic threshold} of $G$, and is denoted by $\chi_{\equiv}(G)$.
The following two conjectures regarding  $\chi_{=}(G)$ and $\chi_{\equiv}(G)$ are still open.

\begin{conj}[Meyer~\cite{M1973}]\label{ECC} For any connected graph $G$, $\chi_{=}(G)\leq \Delta (G)$, with the exception that $G$ is a complete graph or an odd cycle.
\end{conj}

\begin{conj}[Chen, Lih, and Wu~\cite{CLW1994}]\label{EDeltaCC} For any connected graph $G$, $\chi_{\equiv}(G)\leq \Delta(G)$, with the exception that $G$ is a complete graph, an odd cycle, or a
	complete bipartite graph $K_{2m+1,2m+1}$.
\end{conj}

A proper $k$-colouring of $G$ can be viewed as a partition of $V(G)$ into $k$ independent sets.
Lov\'asz~\cite{L1966} considered the corresponding problem where each colour class induces a subgraph of maximum degree at most $d$. Whilst the equitable generalization was first investigated by Fan, Kierstead, Liu, Molla, Wu, and Zhang~\cite{FKLMWZ2011}, where each colour class induces a graph of maximum degree one.

In this paper, we are interested in vertex colouring where each colour class induces a forest. A $k$-colouring of a graph $G$ is a \emph{tree $k$-colouring} if the subgraph of $G$ induced on each colour class is a forest. The minimum such $k$ is called the (\emph{vertex}) \emph{arboricity} of $G$, denoted by $va(G)$. This concept was proposed by Chartrand, Kronk, and Wall~\cite{CKW1968}, who proved that  $va(G)\leq \Delta(G)/2+1$ for any graph $G$ and $va(G)\leq 3$ for any planar graph $G$.
Equitable tree colouring was first studied by Wu, Zhang, and Li~\cite{WZL2013}, who showed (among other results) that every planar graph with girth at least 5 has an equitable tree $k$-colouring for all $k\geq 3$. They further conjectured that the girth condition can be removed but with constant $k$. This conjecture was confirmed by Esperet,  Lemoine, and Maffray~\cite{ELM2015} with $k=4$. For general graphs, Wu, Zhang, and Li~\cite{WZL2013} conjectured the following.

	\begin{conj}[Wu, Zhang, and Li~\cite{WZL2013}]\label{EVAC}
Let $\Delta, k\in \mathbb{N}$ with $k \ge (\Delta+1)/{2}$.
Then every graph with maximum degree at most $\Delta$ has an equitable tree $k$-colouring.
	\end{conj}

	Conjecture~\ref{EVAC} has been verified for several classes of graphs. It holds for graphs with $\Delta(G) \ge |V(G)|/{2}$ by Zhang and Wu~\cite{ZW2014}, for $d$-degenerate graphs $G$ with $\Delta(G) \ge 9.818d$ by Zhang, Niu, Li, and Li~\cite{ZNLL2021}, for all 5-degenerate graphs by Chen, Gao, Shan, Wang, and Wu~\cite{CGSWW2016}, and for graphs $G$ with $\Delta(G)\leq6$ by Chen, Liu, and Zhang~\cite{CLZ2025}.
In this paper, we confirm Conjecture~\ref{EVAC} for large graphs if $\Delta(G)$ is even and up to an additive constant of~$1$ if $\Delta(G)$ is odd.

\begin{thm} \label{thm:treecolouring}
Let $\Delta, k,n \in \mathbb{N}$ with $k \ge (\Delta+2)/{2}$ and $n  \ge3 \Delta^4$.
Then every $n$-vertex graph with maximum degree at most $\Delta$ has an equitable tree $k$-colouring.
\end{thm}

Furthermore, we consider the following generalization of tree colourings. A graph is called $d$-\emph{degenerate} if every subgraph contains a vertex of degree at most $d$. For $d\ge 0$, a $d$-\emph{degenerate} $k$-\emph{colouring} of a graph $G$ is a $k$-colouring such that the subgraph induced on  each colour class is $d$-degenerate.  Note that a $0$-degenerate colouring is a proper colouring, and a 1-degenerate colouring is a tree colouring. Kim, Oum, and Zhang~\cite{KOZ2021} showed that every planar graph has an equitable $d$-degenerate $k$-colouring for any $d=2$, $k=3$ or $d\geq3$, $k=2$. Zhang and Zhang~\cite{ZZ2021} proposed the corresponding statement for tree colouring conjecture.

	\begin{conj}[Zhang and Zhang~\cite{ZZ2021}]\label{conj:EDC}
   Let $\Delta, d, k\in \mathbb{Z}$ with $d \ge 0$ and $k \ge (\Delta+1)/(d+1)$.
    Then every graph with maximum degree at most $\Delta$ has an equitable $d$-degenerate $k$-colouring.
	\end{conj}

Note that the bound in Conjecture~\ref{conj:EDC} is best possible by considering (disjoint union of) $K_{\Delta+1}$. Hajnal and Szemer\'{e}di~\cite{HS1970} resolved the case when $d=0$. Zhang and Zhang~\cite{ZZ2021} confirmed the conjecture for $(d+1)$-degenerate graphs, graphs with bounded maximum average degree, and planar graphs with large maximum degree. We have the following result to support the conjecture.
	
\begin{thm}\label{thm:degeneratecolouring}
    Let $\Delta, d, k\in \mathbb{N}$ with $k \ge \Delta/d+1$.
    Then every graph with maximum degree at most $\Delta$ has an equitable $d$-degenerate $k$-colouring.
	\end{thm}

\subsection{Structure of the paper}

The remainder of this paper is organized as follows. In Section~\ref{sec:property},  we investigate  some properties  of a ``near-equitable''  $d$-degenerate $k$-colouring of $G$, where the  largest colour class and the smallest colour class differ by two or three  in size, and the remaining colour classes have the appropriate sizes.  Based on these properties, we present the proofs of Theorem~\ref{thm:treecolouring} and Theorem~\ref{thm:degeneratecolouring} in Sections~\ref{sec:proof1} and~\ref{sec:proof2}, respectively.

\subsection{Notation}

Let $G$ be a simple graph. For a vertex subset $S \subseteq V(G)$, let $G[S]$ be the subgraph of $G$ induced on $S$ and $G-S=G[V(G)\setminus S]$. We write $G-v$ for   $G-\{v\}$.
Let $N_G(v,S) = N_G(v) \cap S$ and $d_G(v,S) = |N_G(v,S)|$ for $v\in V(G)$.
For $T\subseteq V(G)$, let $N_G(T,S) = \bigcup_{v\in T}N_G(v,S)$.
For two disjoint subsets $S, T \subseteq V(G)$, let $E_G(S,T)$ be the set of edges with one endvertex in $S$ and the other in $T$, and let
$e_G(S,T) = |E_G(S,T)|$.   We write $e_G(v,T)$ for   $e_G(\{v\}, T)$.
For vertices $u,v \in V(G)$, we write \emph{$u \sim v$ in $G$} if $u$ and $v$ are adjacent and  write \emph{$u \not\sim v$ in $G$} otherwise.

\section{Properties of a near-equitable  $d$-degenerate colouring}\label{sec:property}

Our proofs of the main results  follow the framework used by Kierstead and Kostochka~\cite{KK2008} with new ingredients, and  are based on analyzing the properties of a near-equitable
$d$-degenerate colouring.    In this section,  we study  these common properties.
We begin by reviewing basic facts about  $d$-degenerate graphs.

Let $G$ be an $n$-vertex $d$-degenerate graph on vertex set $V:=V(G)$. Then there exists a \emph{degeneracy ordering} of $V$ into $v_1, v_2,\ldots, v_n$ such that each $v_i$ has at most $d$ neighbours in $\{v_1,\ldots,v_{i-1}\}$. Let $p=\max\{0, i \in [n] \colon \text{$v_{i}$ has exactly $d$ neighbours in $\{v_1,\ldots,v_{i-1}\}$}\}$. Note that $p=0$ implies that $G$ is $(d-1)$-degenerate. We will always pick such an ordering with $p$ minimum. Set $V^{*}=\{v_1,\ldots,v_p\}$ if $p>0$ and $V^{*}=\emptyset$ otherwise. For a subset $W\subseteq V$, we define $W^{*}$ to be the corresponding $V^{*}$ for $G[W]$, that is, $W^{*}=(V(G[W]))^{*}$. By the minimality of $p$,
    we have
    \begin{equation}\label{eqn:V^*-vertex-degree}
    d_G(v, V^*) \ge d \quad \text{for  each $v\in V^*$.}
    \end{equation}

The following fact  will be frequently used in  our arguments, sometimes without explicit mention.

\begin{fact}\label{lemma:d-degenerate}~
	\begin{enumerate}[label={\rm(\alph*)}, start=1]
		\item  If $G'$ is obtained from $G$ by joining a new vertex to  at most $d$
		vertices in~$V^{*}$, then $G'$ is $d$-degenerate.\label{d-degenerate-a}
		\item 	 	If $G'$  is obtained
		from $G$ by joining a new vertex to  at most $d-1$
		vertices  of $G$, then $(V(G'))^{*}=V^{*}$.\label{d-degenerate-b}
\end{enumerate}
\end{fact}

\begin{proof}  For~\ref{d-degenerate-a}, placing the new vertex immediately after vertices of $V^{*}$ in the degeneracy ordering gives a $d$-degeneracy ordering of $G'$.

For~\ref{d-degenerate-b}, placing the new vertex at the end of the degeneracy ordering of $G$ gives a $d$-degeneracy ordering of~$G'$. By the minimality of $|V^*|$, we have  $(V(G'))^{*}=V^{*}$.
\end{proof}

Let $G$ be an $n$-vertex graph.
For a given   $d$-degenerate $k$-colouring $\varphi$ of $G$, denote by  $\mathcal{C}^\varphi$ the set of all colour classes of $\varphi$.
For $i\in [k]$, let $C_i^\varphi$ be the $i$-th colour class of $\varphi$ and we assume that
\begin{align*}
|C_1^\varphi| \le |C_2^\varphi| \le \ldots \le |C_k^\varphi|.
\end{align*}
We omit the superscript $\varphi$ if it is clear from context.
We say that $\varphi$ is \emph{near-equitable} if
	\begin{enumerate}[label={\rm(E\arabic*)}]
		\item $\lfloor n/k \rfloor -1 \le |C_1^\varphi| \le \lfloor n/k \rfloor\le |C_2^\varphi| \le \ldots\le |C_{k-1}^\varphi| \le \lceil n/k \rceil  \le |C_k^\varphi| \le \lceil n/k \rceil +1$;\ \text{and}\label{dfn:nearequiable2}
		\item $2 \le |C_k^\varphi| -|C_1^\varphi|  \le 3 $. \label{dfn:nearequiable3}
	\end{enumerate}

	\begin{lem}\label{lem:induction-setting-up}
	Let $\Delta,d,k \in \mathbb{N}$ with  $k \ge (\Delta+1)/(d+1)$.
	Let $G$ be a graph with $\Delta(G)\leq\Delta$ and $uv \in E(G)$.
	If  $G-uv$ has an equitable $d$-degenerate $k$-colouring but $G$ does not,  then $G$ has a near-equitable $d$-degenerate $k$-colouring.
\end{lem}

\begin{proof}
	Let $\psi$ be an equitable $d$-degenerate  $k$-colouring of $G-uv$.
	Since $G$ does not have an equitable $d$-degenerate $k$-colouring, there exists $h \in [k]$ such that $u, v\in C_h^\psi$ and $G[C_h^\psi]$ is not $d$-degenerate.
	Recall that $G[C_h^\psi] -uv$ is  $d$-degenerate, so $d_{G} (v, C_h^\psi) \ge d+1$ (or else we can place $v$ at the end of the degeneracy ordering).
	Note that $\sum_{i \in [k]} d_{G}(v, C_i^\psi) = d_{G}(v) \le \Delta < k(d+1)$.
	Hence there exists $\ell\in [k] \setminus \{h\}$ such that $d_G(v, C_\ell^\psi) \le d$.
	
	Let $\varphi$ be the $k$-colouring obtained from~$\psi$ by recolouring~$v$ with~$\ell$, that is $ \varphi (v) = \ell$ and $\varphi (x) = \psi(x)$ for all $ x \in V(G) \setminus \{v\}$.
	For all $i \in [k]$,
	\begin{align*}
	|C_i^\varphi| =
	\begin{cases}
	  |C_h^\psi| - 1 & \text{if $i = h$,}\\
	  |C_\ell^\psi| + 1 & \text{if $i = \ell$,}\\
	  |C_i^\psi|& \text{otherwise}.
	\end{cases}
	\end{align*}
	Recall that, for all $i \in [k]$, $|C_i^\psi| \in  \{ \lfloor n/k \rfloor, \lceil n/k \rceil\}$.
	By relabelling the colours such that $|C_1^\varphi|,\ldots, |C_k^\varphi|$ is a non-decreasing sequence, $\varphi$ is a desired near-equitable $d$-degenerate $k$-colouring of~$G$.
	Here \ref{dfn:nearequiable3} holds or else $\varphi$ is an equitable $d$-degenerate $k$-colouring of~$G$.
\end{proof}

	Let $G$ be a graph as in Lemma~\ref{lem:induction-setting-up} and $\varphi$ be a near-equitable $d$-degenerate $k$-colouring of $G$.	
	A vertex $v\in V(G)$  is \emph{movable to  a colour class $U \in \mathcal{C}^\varphi$} if
    $G[U\cup \{v\}]$ is $d$-degenerate.
    Clearly if $d_{G}(v,U^*) \le d$, then $v$  is movable to~$U$ by Fact~\ref{lemma:d-degenerate}\ref{d-degenerate-a}.
    We now define an auxiliary digraph~$\mathcal{D}^\varphi$ to monitor movable vertices as follows.
	Define~$\mathcal{D}^\varphi$ to be the digraph  such that $V(\mathcal{D}^\varphi) = \mathcal{C}^\varphi$ and $C_i^\varphi C_j^\varphi \in E(\mathcal{D}^\varphi)$ for distinct $i,j\in [k]$ if and only if there exists a vertex $v \in C_i^\varphi$ that is movable to~$C_j^\varphi$.
	For each $C_i^\varphi C_j^\varphi \in E(\mathcal{D}^\varphi)$, we pick a vertex $v \in C_i^\varphi$ that is movable to~$C_j^\varphi$ and call this vertex \emph{a representative vertex of $C_i^\varphi C_j^\varphi$}.

	Let $\mathcal{P}=U_1\ldots U_\ell$  be a directed path in $\mathcal{D}^\varphi$.
	When we say to \emph{move a vertex} of $U_1$ to $U_\ell$ along~$\mathcal{P}$, we mean that for each $i \in [\ell-1]$, we move a representative vertex of~$U_{i}U_{i+1}$ to~$U_{i+1}$.
	This operation results in a new $d$-degenerate $k$-colouring, where $|U_{\ell}|$ increases  by exactly one, $|U_{1}|$ decreases by exactly one,  and the sizes of remaining colour classes are unchanged.
	Therefore, to obtain an equitable $d$-degenerate $k$-colouring, it suffices to find a directed path from~$C_k$ to~$C_1$.
	
	We say that $V \in \mathcal{C}^\varphi$ is \emph{accessible to~$V' \in \mathcal{C}^\varphi$ in~$\mathcal{D}^\varphi$} if $\mathcal{D}^\varphi$ contains a directed path from~$V$ to~$V'$.
	We say that $V$ is \emph{accessible} if $V$ is accessible to some colour class $U\in \mathcal{C}^\varphi$ of minimum size, that is,  $|U|=|C^\varphi_1|$.
	Let $\mathcal{A}^\varphi$ be the set of accessible colour classes.
	By this definition, all sets of $\mathcal{C}^\varphi$ of size~$|C^\varphi_1|$ are contained in $\mathcal{A}^\varphi$.
	Also $C_k^\varphi \notin \mathcal{A}^\varphi$ (or else we can obtain an equitable $d$-degenerate $k$-colouring).
	Furthermore, let
	\begin{align*}
		A^\varphi &= \bigcup_{V\in \mathcal{A}^\varphi}V, & a^\varphi&= |\mathcal{A}^\varphi|,\\
		\mathcal{B}^\varphi&=\mathcal{C}^\varphi\setminus \mathcal{A}^\varphi, &
		B^\varphi& =\bigcup_{V\in \mathcal{B}^\varphi}V = V(G)\setminus A^\varphi, &
		b^\varphi= |\mathcal{B}^\varphi|.
	\end{align*}
 Note that $a^\varphi+b^\varphi=k$.
Again, we omit superscript $\varphi$ if it is clear from context.
By the definition of~$B^\varphi$, we have
	  \begin{equation}\label{eqn:B-vertex-degree-in-A}
	  d_G(x, V^*) \ge d+1 \text{ for all $x\in B^\varphi$ and all $V\in \mathcal{A}^\varphi$.}
	  \end{equation}

We partition $\mathcal{D}^\varphi[\mathcal{A}^\varphi]$ further.
Suppose that we have already defined vertex-disjoint induced digraphs $\mathcal{D}_1^\varphi, \ldots, \mathcal{D}_{i-1}^\varphi$ of $\mathcal{D}^\varphi[\mathcal{A}^\varphi]$.
Let $j_i=\min\{j\in [k] \colon |C_j| = |C_1|,C_j\in \mathcal{C}^\varphi \setminus \bigcup_{i' \in [i-1]} V(\mathcal{D}_{i'}^\varphi)\} $.
If no such  $j_i$ exists, then we terminate the process.
Let $\mathcal{A}_i^\varphi$ be the set of colour classes in $\mathcal{A}^\varphi \setminus \bigcup_{i' \in [i-1]} V(\mathcal{D}_{i'}^\varphi )$ that are accessible to~$C_{j_i}$.
Let $\mathcal{D}_i^\varphi = \mathcal{D} [\mathcal{A}_i^\varphi]$.
We call $C_{j_i}$ the \emph{terminal} of~$\mathcal{D}_i^\varphi$.

Let $\mathcal{D}_1^\varphi, \ldots, \mathcal{D}_q^\varphi$ be obtained by the algorithm above, which are uniquely defined by~$\varphi$.
By our construction, $\mathcal{D}_i^\varphi$ contains a path from each of its vertices to its terminal  and $\mathcal{A} ^\varphi= \bigcup_{i \in [q]} V(\mathcal{D}_i^\varphi)$.
Most of the time, we are only interested in $\mathcal{D}_q^\varphi$ and its terminal, which we denote by $\mathcal{D}_-^\varphi$ and $C_-^\varphi$, respectively.
Furthermore,
\begin{align}\label{eq:non-accessible}
d_G(v, U^*) \ge d+1 \text{ for any $v\in  V(\mathcal{D}_-^\varphi)$ and $U \in \mathcal{A}^\varphi \setminus V(\mathcal{D}_-^\varphi)$}.
\end{align}

In the next lemma, we show that $|V(\mathcal{D}_-)|\ge 2$, so $a \ge 2$.

\begin{lem}\label{lem:a-lower-bound2}
Let $\Delta, d,k\in \mathbb{N}$ with $k \ge (\Delta+1)/(d+1)$.
Let $G$ be a graph with $\Delta(G)\le \Delta$.
Suppose that $G$ has no equitable $d$-degenerate $k$-colouring, but $G$ has a near-equitable $d$-degenerate $k$-colouring~$\varphi$.
Then $|V(\mathcal{D}_-)|\ge 2$.
\end{lem}

\begin{proof}
Suppose to the contrary that $|V(\mathcal{D}_-)| = 1$, so $ V(\mathcal{D}_-) = \{C_-\}$.
By the definition of~$\mathcal{A}$ and \ref{dfn:nearequiable3}, we have $|V|\ge |C_-|+1$ for any $V\in \mathcal{B}$ and $|C_k| \ge |C_-|+2$.
Thus
\begin{align}
\label{eqn:a-lower}
|B| \ge (b-1)( |C_-| +1)+ |C_-| +2 \ge b |C_-^*|+b+1.
\end{align}
Hence,
 		\begin{align*}
 			\Delta |C_-^*|  &
			\ge  \sum_{v\in C_-^*} d_G(v)  = \sum_{v\in C_-^*} \left( d_G(v,A \setminus C_-) + d_G(v,C_-) + d_G(v,B) \right)\\
			& \ge \sum_{v\in C_-^*} \left( d_G(v,A \setminus C_-) + d_G(v,C_-^*) \right) + e(C_-^*,B)\\
			& \overset{\mathclap{\text{\eqref{eqn:V^*-vertex-degree},\eqref{eqn:B-vertex-degree-in-A},\eqref{eq:non-accessible}}}}\ge
			\quad	\quad ((a-1)(d+1)+d) |C_-^*| +|B|(d+1)\\
 			& \overset{\mathclap{\text{\eqref{eqn:a-lower}}}} \ge  \quad ((a+b)(d+1)-1) |C_-^*| +(b+1)(d+1)
 			= (k(d+1)-1) |C_-^*|+(b+1)(d+1) \\
			 &\ge \Delta |C_-^*| + (b+1)(d+1),
 		\end{align*}
a contradiction.
 \end{proof}

A vertex $V\in V(\mathcal{D}_-)$ is a \emph{terminal-cut} of $\mathcal{D}_-$ if $\mathcal{D}_--V$ has a vertex that is non-accessible to~$C_-$ in~$\mathcal{D}_--V$.
Since $|V(\mathcal{D}_-)| \ge 2$ by Lemma~\ref{lem:a-lower-bound2}, $C_-$ is a terminal-cut.
Let $U_- \in V(\mathcal{D}_-)$ be such that the number of vertices in $\mathcal{D}_- -U_-$ that are non-accessible to $C_-$ is minimal.
(If there are multiple choices for~$U_-$, then  we fix one arbitrarily.)
We let $\mathcal{T}^\varphi$ be the set of vertices of~$\mathcal{D}_- -U_-$  that are
non-accessible to~$C_-$ in~$\mathcal{D}_- -U_-$.
Let
\begin{align*}
T^\varphi=\bigcup_{V\in \mathcal{T}^\varphi} V\quad
 \text{ and }\quad
 t^\varphi=|\mathcal{T}^\varphi|.
\end{align*}

We say that a near-equitable $d$-degenerate $k$-colouring~$\varphi$ of~$G$ is \emph{minimal} if
\begin{align*}
	( |C_k^\varphi|-|C_-^\varphi|,\, b^\varphi,\,   \sum\limits_{V\in \mathcal{T}^\varphi} |V^*|)
\end{align*}
is  lexicographically smallest over all near-equitable $d$-degenerate $k$-colourings of~$G$.

We are now ready to establish several key structural properties of~$\varphi$. These properties describe how vertices may move among colour classes in a minimal near-equitable $d$-degenerate colouring and provide useful bounds on the sizes of the sets in~$\mathcal{A}$ and~$\mathcal{B}$.

In Lemma~\ref{lemma:common-properties}, statement~\ref{a} shows that each colour class in~$\mathcal{A}$ has size either $|C_-|$ or $|C_-|+1$. Statement~\ref{b} provides a lower bound on~$|B|$. Statements~\ref{c}-\ref{g} describe structural properties of the colour classes in~$\mathcal{T}$. In particular, statement~\ref{e} identifies a colour class~$W$, which plays a central role in the proofs of our main theorems. Finally, the bound on~$t$ in statement~\ref{h} is crucial for deriving a contradiction through edge-counting arguments in the proof of Theorem~\ref{thm:treecolouring}.
	
	\begin{lem}\label{lemma:common-properties}
	Let $\Delta, d,k \in \mathbb{N}$ with $k \ge (\Delta+1)/(d+1)$, and $G$ be an $n$-vertex graph with $\Delta(G)\leq\Delta$.
		Suppose that $G$ has no equitable $d$-degenerate $k$-colouring, but $G$ has a minimal near-equitable $d$-degenerate $k$-colouring~$\varphi$.
	Then the following statements hold:
\begin{enumerate}[label={\rm (\alph*)},start=1]
\item  for all $V\in \mathcal{A}$,  $|V|-|C_-| \le 1$ and $V^*\neq \emptyset$; \label{a}
\item  for all $V\in \mathcal{A}$, $|B| \ge b |V|+1$;
 \label{b}		
\item for all $V\in \mathcal{T}$, each $v \in V$ is non-movable to any colour class in $\mathcal{A}\setminus ( \mathcal{T}\cup \{U_-\})$, in particular, $d_G(v,A \setminus (T \cup U_-)) \ge (d+1)(a-1-t)$; \label{c}
\item for all distinct $V,V'\in \mathcal{T}$, $\mathcal{D}-V$ contains a directed $(V',C_-)$-path; \label{d}
\item there exists $W\in \mathcal{T}$ such that for all $w\in W^{*}$ and $V\in \mathcal{T}$, we have $d_G(w,V) \ge d$.\label{e}	
\end{enumerate}			
Furthermore,  suppose that for all $x \in B$, $G[B]-x$ has an equitable $d$-degenerate $b$-colouring.
Then the following statements hold:	
\begin{enumerate}	[label={\rm (\alph*)},start=6]		
\item
if $v \in V^*$, $V  \in \mathcal{T}$ and $x \in N_G(v) \cap B$ such that $d_G(x,V^*) =d+1$, then $v$ is non-movable to any other colour class in~$\mathcal{A}$, in particular, $d_G(v,A) \ge (d+1)a-1$;
	  \label{f}
\item   if $ v \in V \in \mathcal{T}$ and there exist distinct $x,y \in B\cap N_{G}(v)$ such that
$d_G(x,V)  = d_G(y,V) =d+1$, then $x \sim y$ in~$G$; \label{g}		
\item   if $\tau \le d+1$ and $ k \ge \frac{1+\tau}{\tau} \frac{\Delta+1}{d+2}$, then $t < \tau b$. \label{h}
\end{enumerate}
\end{lem}

	\begin{proof}	
	Proof of~\ref{a}: For the first part, suppose otherwise that $|V|-|C_-|  \ge 2$ for some $V\in \mathcal{A}$. Note that $C_k \notin \mathcal{A}$.
	Since $V \in \mathcal{A}$, \ref{dfn:nearequiable2} implies that
	\begin{align*}
	\text{$|V|  =  \lceil n/k \rceil$,  $n \not \equiv 0 \bmod{k}$,  $|C_-| = \lfloor n/k \rfloor-1$, and $C_- = C_1$.}
	\end{align*}
	By the definition of $\mathcal{A}$, there exists a directed $(V,U)$-path~$\mathcal{P}$ for some $U \in \mathcal{A}$ with $|U| = |C_-|$.
	Define $\psi$ to be the $d$-degenerate $k$-colouring obtained from $\varphi$ by moving a vertex of~$V$ to~$U$ along $\mathcal{P}$.
	
	If $|C_k| = \lceil n/k \rceil$, then $\psi$ is an equitable $d$-degenerate $k$-colouring of~$G$, a contradiction.
	We may assume that $|C_k| = \lceil n/k \rceil+1$ by~\ref{dfn:nearequiable2} and so $|C_k| - |C_-| = 3$.
	Note that $\psi$ is a near equitable $d$-degenerate $k$-colouring with $|C_k^\psi| - |C_-^\psi| = 2$, contradicting the minimality of~$\varphi$.
	
For the second part, suppose that $V^*=\emptyset$ for some  $V\in \mathcal{A}$.
Then $G[V]$ is $(d-1)$-degenerate. Thus, $C_k$ is accessible to $V$ by Fact~\ref{lemma:d-degenerate}(a) and is also accessible to $C_-$, contradicting $C_k\in \mathcal{B}$.
	
	Proof of~\ref{b}:
	Consider $V \in \mathcal{A}$.
	By~\ref{dfn:nearequiable3} and~\ref{a}, $|C_k| \ge |C_-|+2 \ge  |V| +1$.
	By the definition of~$\mathcal{A}$ and~\ref{a}, for all $U \in \mathcal{B}$, we have $|U| \ge |C_-| +1 \ge |V|$.
	Hence, \ref{b} follows.

		Proof of~\ref{c}:
	Suppose to the contrary that there exists a vertex $v \in V\in \mathcal{T}$ that is movable to some $V'\in \mathcal{A}\setminus (\mathcal{T}\cup \{U_-\}) $.
	Hence $VV' \in E(\mathcal{D})$.
	Note that $V' \in V(\mathcal{D}_-)$ by our definition of $\mathcal{D}_1, \ldots, \mathcal{D}_q$ with $\mathcal{D}_q = \mathcal{D}_-$.
	Thus there is a directed  $(V', C_-)$-path in~$\mathcal{D}_- - U_-$ as $V' \notin \mathcal{T}$.
	Since $VV' \in E(\mathcal{D})$, there is a directed $(V, C_-)$-path in~$\mathcal{D}_- - U_-$,  a contradiction to the assumption that $V\in \mathcal{T}$.

	Proof of~\ref{d}: Suppose to the contrary that there is no directed $(V',C_-)$-path in $\mathcal{D}-V$ for some distinct $V,V' \in \mathcal{T}$.
	Recall that $V,V' \in V(\mathcal{D}_-)$.
	Thus $V$ is a terminal-cut of~$\mathcal{D}_-$.
	Let $\mathcal{T}'$ be the set of vertices of~$\mathcal{D}_- - V$  that are
non-accessible to~$C_-$ in~$\mathcal{D}_- - V$.
	Since $V \in \mathcal{T}$, all directed $(V,C_-)$-paths go through~$U_-$ and so $U_- \not\in \mathcal{T}'$.
	Also if a colour class $U \in  V(\mathcal{D}_-) \setminus \{U_-\}$ is accessible to $C_-$ in~$\mathcal{D}_- - U_-$, then it is accessible to $C_-$ in~$\mathcal{D}_- - V$.
	Hence $\mathcal{T}' \subsetneq \mathcal{T}$, contradicting the minimality of $|\mathcal{T}|$.

	Proof of~\ref{e}:
	For all $V \in \mathcal{T}$ and all $ v \in V^*$, we have $d_G(v,V) \ge d$ by~\eqref{eqn:V^*-vertex-degree}.
	Let $\mathcal{D}^*$ be the digraph on $\mathcal{T}$ 	such that $VU\in E(\mathcal{D}^*)$ if there exists $v\in V^*$ such that $d_G(v,U) \le d-1$.
	If $\mathcal{D}^*$ is acyclic, then it has a vertex $W$ that has outdegree zero with the desired property.
	
	Hence we may assume that $\mathcal{D}^*$ contains a cycle $C_{i_1}\ldots C_{i_p}C_{i_1}$.
	For each $j \in [p]$, pick $v_j \in C_{i_j}^*$ such that  $d_G(v_{j},C_{i_{j+1}}) \le d-1$, where we take $i_{p+1} = i_1$.
	Define $\psi$ to be the $k$-colouring obtained by moving~$v_{j}$ to~$C_{i_{j+1}}$ for all $j \in [p]$.
	By Fact~\ref{lemma:d-degenerate}\ref{d-degenerate-a}, $\psi$ is a $d$-degenerate $k$-colouring of~$G$.
	We have $|C_i^\psi| = |C_i^\varphi|$ for all $i \in [k]$, so $\psi$ is near-equitable.
	By Fact~\ref{lemma:d-degenerate}\ref{d-degenerate-b},
	\begin{align}
	\text{$(C_{i}^{\psi})^* \subseteq (C_{i}^{\varphi})^*$ for all $i \in [k]$, with equality holding if and only if $i \notin \{i_j:j \in [p]\}$.} \label{eqn:d}
	\end{align}
	
    \begin{claim}\label{claim:Ci}
    If $C_i^\varphi$ is accessible to~$C_j^\varphi$ in~$\mathcal{D}^\varphi$, then $C_i^\psi$ is also accessible to~$C_j^\psi$ in~$\mathcal{D}^\psi$.
    \end{claim}

    \begin{proofclaim}
		It suffices to show that for any edge $C_i^\varphi C_{i'}^\varphi \in E(\mathcal{D}^\varphi)$, there exists a directed walk from~$C_i^\psi$ to~$C_{i'}^\psi$ in $\mathcal{D}^\psi$.
Let $v$ be a representative vertex of $C_i^\varphi C_{i'}^\varphi$.

First we show that $v$ is movable to~$C_{i'}^\psi$.
If $i' \notin \{i_j\colon j \in [p]\}$, then this clearly holds as $C_{i'}^\psi = C_{i'}^\varphi$.
If $i' = i_j$ for some $j \in  [p]$, then $C_{i_j}^\psi = (C_{i_j}^\varphi \setminus \{v_j\}) \cup \{v_{j-1}\}$, where $v_0 = v_p$.
Since $v$ is a representative vertex of~$C_i^\varphi C_{i'}^\varphi$, we have $G[C_{i'}^{\varphi} \cup \{v\}]$ is $d$-degenerate.
  Note that
    \begin{align*}
        d_G(v_{j-1},C^{\varphi}_{i'} \cup \{v\}) \le d_G(v_{j-1},C^{\varphi}_{i'})+1\le d,
    \end{align*}
    so $G[C_{i'}^{\psi} \cup \{v,v_{j-1}\}]$ is $d$-degenerate by Fact~\ref{lemma:d-degenerate}\ref{d-degenerate-a}.
Hence $v$ is movable to $C_{i'}^\psi$.

Therefore if $ v \in C_i^\psi$, then $C_i^\psi C_{i'}^\psi \in E(\mathcal{D}^\psi)$.
Suppose that $v \notin C_i^\psi$, so $v = v_{j_0}$ for some $j_0 \in [p]$.
Moreover, $i = i_{j_0}$ and $v \in C^{\psi}_{i_{j_0+1}}$.
Since for all $j \in [p]$, we have $v_j\in C_{i_{j+1}}^\psi$ and $G[C_{i_j}^{\psi} \cup \{v_j\}] = G[C_{i_j}^{\varphi} \cup \{v_{j-1}\}]$ is $d$-degenerate, we deduce that $C^{\psi}_{i_p}C^{\psi}_{i_{p-1}}\ldots C^{\psi}_{i_1}C^{\psi}_{i_p}$ is a directed cycle in~$\mathcal{D}^{\psi}$.
Therefore, $C_{i_{j_0}}^{\psi} C_{i_{j_0-1}}^\psi \ldots C_{i_{{j_0+1}}}^{\psi} C_{i'}^\psi$ is a directed walk in $\mathcal{D}^\psi$, where the sub-indices are taken modulo $p$.
This implies that $C_{i}^\psi$ is accessible to $C_{i'}^\psi$, as required.
    \end{proofclaim}

By Claim~\ref{claim:Ci}, we have  $\mathcal{A}^\varphi \subseteq \mathcal{A}^\psi $.
By the minimality of~$\varphi$, we have $\mathcal{A}^\varphi = \mathcal{A}^\psi $.
However, \eqref{eqn:d} implies that
\begin{align*}
	\sum_{V \in \mathcal{T}^\psi}|V^*| \leq \sum_{ V \in \mathcal{T}^\varphi}|V^*|-p,
\end{align*}
contradicting the minimality of~$\varphi$.

	Proof of~\ref{f}: Suppose to the contrary that there exist $v \in V^*$, $V \in \mathcal{T}$ and $x \in N_G(v) \cap B$ such that $d_G(x,V^*) =d+1$ and $v$ is movable to some colour class~$V' \in \mathcal{A}$.
	By~\ref{c}, $V' \in \mathcal{T} \setminus \{U_-\}$.
	By~\ref{d}, there is a directed $(V',C_-)$-path~$\mathcal{P}$ in $\mathcal{D}-V$.
	Now we move a vertex of $V'$ to $C_-$ along $\mathcal{P}$.
	Note that $v$ remains in $V$.
	We now move $x$ to~$V$ and $v$ to~$V'$.
	This yields an equitable $d$-degenerate $a$-colouring~$\psi_1$ on~$G[A \cup \{x\}]$.
	By the assumption, $G[B]-x$ has an equitable $d$-degenerate $b$-colouring~$\psi_2$.
	Combining $\psi_1$ and~$\psi_2$, we obtain an equitable $d$-degenerate $k$-colouring of~$G$, a contradiction.
	
	Proof of~\ref{g}:
	Consider $ v \in V \in \mathcal{T}$ and distinct $x,y \in B\cap N_{G}(v)$ such that
$d_G(x,V)  = d_G(y,V) =d+1$.
    Suppose to the contrary that $x\not\sim y$.

    By~\eqref{eqn:B-vertex-degree-in-A}, we have $d_G(x,V^*)  =d+1$ and so $v \in V^*$.
	  By~\ref{f}, we have
    \begin{align*}
    d_G(v,B \setminus \{x\}) & = d_G(v) - 1-d_G(v,A) \le \Delta -1- ((d+1)a - 1)\\
    & \le (d+1)k-1 - (d+1)a = (d+1)b-1.
    \end{align*}
    By our assumption let $\psi_B$ be an equitable $d$-degenerate $b$-colouring of $G[B]- x$.
    By~\ref{dfn:nearequiable2}, we deduce that each colour class of~$\psi_B$ has size at most $\lceil n/k \rceil$.
    Since $v\sim x$ and $d_G(v,B \setminus \{x\})\le (d+1)b-1$, there exists a colour class~$U$ in~$\psi_B$ such that $d_G(v,U) \le d$.
    Define $\psi_B'$ to be a $d$-degenerate $b$-colouring of $G[(B \setminus \{x\}) \cup \{v\}]$ by adding~$v$ to~$U$.
    Let $\psi_A$ be the $d$-degenerate colouring of~$G[(A \setminus \{v\}) \cup \{x\}]$ obtained from $\varphi$ restricted on~$A$ by swapping~$v$ with~$x$.
    That is, $\psi_A(u) = \varphi(u)$ for all $u \in A \setminus \{v\}$ and $\psi_A(x) = \varphi(v)$.
    Combine~$\psi_A$ and~$\psi_B'$ to obtain a $d$-degenerate $k$-colouring~$\psi$ of~$G$.
    Since $G$ does not have an equitable $d$-degenerate $k$-colouring, we deduce that $\psi$ is near-equitable.

Let $C_i^{\varphi}= V$ be the colour class containing~$v$,    so $C_i^{\psi} = (C_i^{\varphi} \setminus \{v\}) \cup \{x\}$.
By~\ref{d}, we deduce that  $\mathcal{A}^{\varphi} \setminus \{C_i^{\varphi} \} \subseteq \mathcal{A}^{\psi} \setminus \{C_i^{\psi} \}$.
Since  $C_i^{\varphi} \in  \mathcal{A}^{\varphi}$, there exists a vertex $z \in C_i^{\varphi}$  that is movable to some $U \in \mathcal{A}^{\varphi} \setminus \{C_i^{\varphi}\} = \mathcal{A}^{\psi} \setminus \{C_i^{\psi}\}$.
Note that $z\neq x$, so $z \in C_i^\psi$ and $C_i^{\psi} \in  \mathcal{A}^{\psi}$.
Since $d_{G}(y,C_i^{\psi}) =  d_G(y,C_i^{\varphi}) -1 =d$, $y$ is movable to~$C_i^{\psi}$.
However $y \in B^{\varphi}$ and $C_i^{\psi} \in \mathcal{A}^{\psi}$, so we have $\mathcal{A}^{\varphi} \subsetneq \mathcal{A}^{\psi}$, contradicting the minimality of~$\varphi$.

	Proof of~\ref{h}:  	 Suppose to the contrary that $t  \ge   \tau b$.
	Then  $k =a+b \ge t+b \ge  (1+\tau)b$ and so
	\begin{align*}
	b \le \frac{k}{1+\tau}\quad \text{ and }\quad	a  = k-b  \ge  \frac{\tau}{1+\tau} k.
	\end{align*}
Hence
\begin{align}
\tau b + (d+1) a 	&=  \tau k + (d+1-\tau) a
	\ge  k\left(\tau + (d+1-\tau)\frac{\tau}{1+\tau}\right) \nonumber \\
	&= \frac{(d+2)\tau}{1+ \tau} k
	 \ge  \frac{(d+2)\tau}{1+ \tau}  \frac{1+\tau}{\tau(d+2)} (\Delta+1)
	= \Delta+1.      \label{eqn:taub}
\end{align}

For any $x \in B$ and $v\in V\in \mathcal{T}$, let
	\begin{align*}
		\mathcal{T}(x) =\{V\in \mathcal{T}\colon d_G(x, V)=d+1\}\quad \text{ and }\quad
		B(v) = \{x\in  N_G(v,B)\colon V\in  \mathcal{T}(x)\}.
	\end{align*}
We now define a weight function~$\mu : B \times T  \rightarrow \mathbb{R}$ such that  for any $x\in B$  and $v\in T$,
\begin{align*}
	\mu(x,v)=  \begin{cases}
			\frac{b}{|B(v)|} & \text{if $x \in B(v)$,}\\
			0 & \text{otherwise.}
	\end{cases}
\end{align*}
Note that
\begin{align}
\sum_{x \in B, v\in T} \mu(x,v) = \sum_{v \in T} \sum_{x \in B(v)} \mu(x,v)
= \sum_{v \in T \colon B(v) \ne \emptyset} b
\le \sum_{V \in \mathcal{T}} b |V|
\overset{\mathclap{\text{\ref{b}}}}\le t( |B|-1). \label{eqn:muupper}
\end{align}
Consider $x \in B$.	Note that
	\begin{align}
		|\mathcal{T}(x)|
		& \overset{\mathclap{\text{\eqref{eqn:B-vertex-degree-in-A}}}}{\ge}t-(d_{G}(x,A)-(d+1)a)=   t+(d+1)a  - d_{G}(x,A)\nonumber\\
	&	\ge t+(d+1)a  + d_G(x,B) - \Delta
  \overset{\mathclap{\text{\eqref{eqn:taub}}}}{\geq} t-\tau b+d_{G}(x,B)+1 >0, \label{eq:Tx}
	\end{align}
		where the last inequality holds as $ t \ge \tau b$.
If $v \in V \subseteq T$ with $x \in B(v)$, then $d_G(x,V) = d+1$.
Hence $G[B(v)]$ is a clique by~\ref{g} and  so
	\begin{align}
		|B(v)| \leq d_{G}(x,B)+1.
		\label{eqn:B(v)}
	\end{align}
	Also, \eqref{eqn:B-vertex-degree-in-A} implies that
	\begin{align}
		 d_{G}(x,B)+1  = d_{G}(x) - d_{G}(x,A)+1 \le \Delta+1 - a(d+1) \le b(d+1).
		\label{eqn:d(x,B)+1}
	\end{align}
Therefore,
\begin{align*}
	\sum_{v \in T }\mu(x,v) & =
    \sum_{v \in T \colon x \in B(v)}\mu(x,v) = \sum_{ V \in \mathcal{T}(x)} \sum_{v\in N_G(x,V)} \frac{b}{|B(v)|}
		\overset{\mathclap{\text{\eqref{eqn:B(v)}}}}{\geq}  |\mathcal{T}(x)| (d+1) \frac{b}{d_{G}(x,B)+1} \\
		&\overset{\mathclap{\text{\eqref{eq:Tx}}}}{\geq}\quad (t-\tau b+d_{G}(x,B)+1)  \frac{b(d+1)}{d_{G}(x,B)+1}
	 = \frac{b(d+1)(t-\tau b)}{d_{G}(x,B)+1}+b(d+1) \\
	&\overset{\mathclap{\text{\eqref{eqn:d(x,B)+1}}}}{\geq}\quad\frac{b(d+1)(t-\tau b)}{b(d+1)}+b(d+1)
	=  t+b(d+1 - \tau) \ge t,
\end{align*}
where the last inequality holds as $\tau \le d+1$.
This implies that $\sum_{x \in B, v\in T} \mu(x,v) \ge t|B|$, contradicting~\eqref{eqn:muupper}.
	\end{proof}

\section{Equitable $d$-degenerate colourings}\label{sec:proof2}

We now prove Theorem~\ref{thm:degeneratecolouring} using Lemmas~\ref{lem:induction-setting-up} and~\ref{lemma:common-properties}.

\begin{proof}[Proof of Theorem~\ref{thm:degeneratecolouring}]
We apply induction on $e(G)$. The statement  holds trivially if $e(G)=0$.
		Thus we assume $e(G) \ge 1$.
		Let $u_0v_0\in E(G)$.
		By the induction hypothesis, $G-u_0v_0$ has an equitable $d$-degenerate $k$-colouring. We may assume that $G$ has no
		equitable $d$-degenerate $k$-colouring.
		Applying Lemma~\ref{lem:induction-setting-up},
		$G$ has a near-equitable $d$-degenerate $k$-colouring~$\varphi$.
		We may further assume that $\varphi$ is minimal.

Let $W \in \mathcal{T}$ be as given by Lemma~\ref{lemma:common-properties}\ref{e}.
Together with Lemma~\ref{lemma:common-properties}\ref{c}, we have
\begin{align*}
	\sum_{w \in W^*} d_G(w,A) & \ge \sum_{w \in W^*} \left( d_G(w,A \setminus (T \cup U_-)) +  d_G(w, T) \right)\\
	& \ge \left((d+1)(a-1-t) + t d\right) |W^*|  = ((a-1)d + a-1-t )|W^*| \ge (a-1)d|W^*|
\end{align*}
as $ t < a$.
On the other hand, recall that $\Delta \le d(k-1) =d(a+b-1)$.
We have
\begin{align}\label{eqn:e(B,S)}
	e_G(B,W^*) & = \sum_{w \in W^*} \left( d_G(w)- d_G(w,A) \right) \le \Delta |W^*| - 	\sum_{w\in W^*} d_G(w,A) \nonumber \\
		& \le
		d(a+b-1)|W^*|- 	\sum_{w\in W^*} d_G(w,A)
		\le db |W^*|
		\overset{\mathclap{\text{Lemma~\ref{lemma:common-properties}\ref{b}}}}{<} d|B|.
\end{align}

Now we give a lower bound on $e_G(B,W^*)$.
Let $B_1=\{x\in B\colon d_G(x,W^*) =d+1\}$, so $|B_1| =e_G(B_1,W^*)/(d+1)\le e_G(B,W^*)/(d+1)$.
Then by~\eqref{eqn:B-vertex-degree-in-A},
\begin{align*}
e_G(B,W^*)\ge (d+2)|B|-|B_1|
\ge (d+2)|B|-\frac{1}{d+1}e_G(B,W^*).
\end{align*}
Thus, $e_G(B,W^*)\geq (d+1)|B|$, contradicting~\eqref{eqn:e(B,S)}.
\end{proof}

\section{Equitable tree colourings}\label{sec:proof1}

Recall that an equitable tree colouring is also an equitable 1-degenerate colouring. We now sketch the proof of Theorem~\ref{thm:treecolouring}.
Suppose that $G$ satisfies all the hypotheses of Lemma~\ref{lemma:common-properties}.
Let $\varphi$ be a minimal near-equitable $1$-degenerate $k$-colouring of~$G$ and
$W \in \mathcal{T}$ be as given by Lemma~\ref{lemma:common-properties}\ref{e}.
An underlying goal will be to show that $t \ge 2b$, which will then contradict Lemma~\ref{lemma:common-properties}\ref{h}.
This will be achieved by bounding $e_G(B,W^*)$ from both above and below.

Most of our new ideas are needed to bound $e_G(B,W^*)$ from below as follows.
We do so by counting the number of edges from~$B$ to~$W^*$.
By~\eqref{eqn:B-vertex-degree-in-A}, every vertex in~$B$ has at least two neighbours in~$W^*$.
Let $B_1$ and $B_2$ be the sets of vertices in~$B$ that have exactly two and three neighbours in~$W^*$, respectively.
Hence we have
\begin{align*}
	e_G(B,W^*) \ge 2 |B_1| + 3 |B_2| + 4 |B \setminus (B_1 \cup B_2)| = 4|B| - 2|B_1| - |B_2|.
\end{align*}
Thus we would need to upper bound $|B_1|$ and~$|B_2|$.

Consider $x \in B_1$ and $w \in N_G(x,W^*)$.
By Lemma~\ref{lemma:common-properties}\ref{f}, $w$ is non-movable to any other colour class in~$\mathcal{A}$.
(Otherwise, we can move~$w$ to $C_-$ along some directed path in~$\mathcal{D}$ and then we can move~$x$ into~$W$.)
Let $W_1$ be the set of vertices in~$W^*$ that are non-movable to any other colour class in~$\mathcal{A}$.
Hence, we have $N_G(B_1, W^*) \subseteq W_1$ and so $2|B_1| = e_G(B_1,W_1) \le e_G(B,W_1)$.

Next consider $x \in B_2$, so $d_G(x,W^*)=3$.
If $x$ has two neighbours in~$W_1$, then we can use~$e_G(B,W_1)$ to bound the number of such vertices.
Hence we assume that $x$ has two neighbours $w_1, w_2$ in~$W^* \setminus W_1$.
First, we move~$w_1$ to some colour class $V_1\in \mathcal{T}\cup\{U_-\}$ and some vertex in $V_1$ to $C_-$ along some directed path in~$\mathcal{D}$.
If we are able to move $x$ to~$W \setminus \{w_1\}$, then we get a contradiction.
Hence assume that $G[W \cup \{x\} \setminus \{w_1\}]$ is not $1$-degenerate, that is, not acyclic.
All cycles in  $G[W \cup \{x\} \setminus \{w_1\}]$ must contain both $x$ and $w_2$.
To remove~$w_2$ from~$W$, we find a colour class $V \in \mathcal{T}$ such that $w_2$ is movable to~$V$ and there is a vertex $v \in V$ that is movable to~$W$.
We then swap $w_2$ and~$v$ by moving $w_2$ to~$V$ and $v$ to~$W$.
We can now move $x$ into~$W$ (providing $N_G(v) \cap ( \{ x \} \cup N_G(x)) = \emptyset$), see Claim~\ref{clm:d=1:B} for details.
To identify suitable colour classes for this swapping operation, we seek colour classes~$V$ such that
$e_G(W^*,V) < 2(|V|-\Delta^2)$.
The set of such colour classes will be denoted by~$\mathcal{R}$.

\begin{proof}[Proof of Theorem~\ref{thm:treecolouring}]
Let $G$ be an $n$-vertex graph with $\Delta(G) \le \Delta$.
If $k \ge \Delta +1$, then by the Hajnal--Szemer\'{e}di theorem~\cite{HS1970},
$G$ contains a  proper equitable $k$-colouring.
Hence we may assume that $\Delta \ge k$.
	
We prove the theorem with a weaker assumption that $\lfloor n/k \rfloor \ge 3\Delta^3$.
(Since $\Delta \ge k$ and $n \ge 3 \Delta^4$, we have $\lfloor n/k \rfloor \ge 3\Delta^3$.)
We apply induction on $e(G)$. The statement  holds trivially if $e(G)=0$.
		Thus we assume $e(G) \ge 1$.
		Let $u_0v_0\in E(G)$.
		By the induction hypothesis, $G-u_0v_0$ has an equitable $1$-degenerate $k$-colouring.
		We may assume that $G$ has no	equitable $1$-degenerate $k$-colouring.
		Applying Lemma~\ref{lem:induction-setting-up}, $G$ has a near-equitable $1$-degenerate $k$-colouring~$\varphi$.
		We may further assume that $\varphi$ is minimal.
	
Note that we have
	\begin{align}
     \Delta+2  &\le  2k = 2a+2b.
           \label{eqn:d=1:Delta1}
        \end{align}
By~\eqref{eqn:B-vertex-degree-in-A}, we have for all $x \in B$,
	\begin{align*}
		d_G(x, B) &= d_G(x) - d_G(x, A) \le \Delta - 2a  \overset{\mathclap{\text{\eqref{eqn:d=1:Delta1}}}}{\le} 2b-2.
	\end{align*}
    Hence $b \ge (\Delta(G[B])+2)/2$ and $e(G[B]) < e(G)$.
		Consider $V \in \mathcal{T}$.
		By Lemma~\ref{lemma:common-properties}\ref{b}, we have
    \begin{align*}
    \left\lfloor \frac{|B|-1}b  \right\rfloor \ge |V| \ge \left\lfloor \frac{n}{k} \right\rfloor \ge 3\Delta^3 \geq 3 (\Delta(G[B]))^3.
    \end{align*}
	Hence by our induction hypothesis, $G[B ]-x$ has  an equitable $1$-degenerate $b$-colouring.
Thus $\varphi$ satisfies all the properties listed in Lemma~\ref{lemma:common-properties}.

Let $W \in \mathcal{T}$ be as given by Lemma~\ref{lemma:common-properties}\ref{e}.
Next, we define $\mathcal{R}$ to be the set of colour classes in~$\mathcal{T}$ that send few edges to~$W^*$ as follows.
Let
\begin{align*}
 \mathcal{R}&=\{V\in \mathcal{T}  \setminus \{W\} \colon  e_G(V, W^*) < 2(|V| -  \Delta^2) \}, \qquad
 R =\bigcup_{V\in \mathcal{R}} V, \qquad r =|\mathcal{R}|, \\
	\mathcal{R}' & = \mathcal{R} \cup \{U_-\}, \qquad R' =\bigcup_{V\in \mathcal{R}'} V = R \cup U_-.
\end{align*}
We now further partition~$W$ as follows.
Let
\begin{align*}
		W_1 &=\{w\in W^* \colon \text{$d_G(w,V^*) \ge 2 $ for any $V \in \mathcal{A} \setminus\{W\}$}\}, \\
  W_2 &=\left\{w \in W^* \setminus W_1 \colon d_G(w, V^*)\le 1 \text{ for exactly one set $V\in \mathcal{R}'$} \right\}, \\
		W_{>2} &= W^* \setminus  (W_1 \cup W_2).
\end{align*}
Intuitively, we view $W_1$ to be the set of vertices in $W^*$ that is not movable to any other colour classes, $W_2$ to be those vertices that are movable to exactly one colour class in~$\mathcal{R}'$.
Next we bound the number of edges between~$B$ and~$W^*$ from above.

\begin{claim}\label{claim:d=1upper}
The following inequalities hold.
\begin{enumerate}[label={\rm(\alph*)},start=1]
	\item $ e_G(B,W_1)  \le 	(2b-1)|W_1|$. \label{itm:d=1:e(B,W1)}
	\item $e_G(B,W_{2})   \le   (2b+t-r) |W_{2}|$. \label{itm:d=1:e(B,W2)}
	 \item $e_G(B,W^*)    \le (2b+1) |W|  +  \frac23|W_2|+(r+1)|W_{>2}|$. \label{itm:d=1:e(B,W3)}
\end{enumerate}
\end{claim}

\begin{proofclaim}
For any vertex subset $S \subseteq A$, we have
\begin{align}
	e_G(B,S) & = \sum_{w\in S} \left( d_G(w)- d_G(w,A) \right) \le \Delta |S| - 	\sum_{w\in S} d_G(w,A)
		\overset{\mathclap{\text{\eqref{eqn:d=1:Delta1}}}}{\le}
		2(a+b-1)|S|- 	\sum_{w\in S} d_G(w,A).
		\label{eqn:d=1:u1}
\end{align}
To prove the claim, we bound $\sum_{w \in S} d_G(w,A)$ from below for $S \in \{W_1,W_2, W^*\}$ as follows.

Recall that $\delta(G[W^*]) \ge 1$ by~\eqref{eqn:V^*-vertex-degree} and $W_1 \subseteq W^*$.
By~\eqref{eqn:V^*-vertex-degree} and the definition of~$W_1$, we have
\begin{align*}
	\sum_{w \in W_1} d_G(w,A) & \ge \sum_{w \in W_1} \left( 2a - 1 \right)=\left( 2a - 1 \right) |W_1|.
\end{align*}
Together with~\eqref{eqn:d=1:u1} taking~$S=W_1$, \ref{itm:d=1:e(B,W1)} holds.

Consider $w \in W_2$.
By Lemma~\ref{lemma:common-properties}\ref{c}, we have $d_G(w,A \setminus (T \cup U_-)) \ge 2(a-t-1)$. By Lemma~\ref{lemma:common-properties}\ref{e}, we have $d_G(w,V) \ge 1$ for all $V \in \mathcal{T}$. Combine with~\eqref{eqn:V^*-vertex-degree} and the definition of~$W_2$, we have
\begin{align*}
	\sum_{w \in W_2} d_G(w,A) & \ge \sum_{w \in W_2} \left( d_G(w,A \setminus (T \cup U_-)) + d_G(w,R') +d_G(w, T \setminus R ) \right) \\
	& \ge \sum_{w \in W_2} \left( 2(a-t-1)  + 2r 	+  t-r\right)  = ( 2a -t + r -2 ) |W_2|.
\end{align*}
Together with~\eqref{eqn:d=1:u1} taking~$S=W_2$, \ref{itm:d=1:e(B,W2)} holds.

We now prove \ref{itm:d=1:e(B,W3)}.
Let $T' = T \setminus (R \cup W)$.
We consider $\sum_{w \in W^*}d_G(w,A \setminus T')$ and $\sum_{w \in W^*}d_G(w,T')$ separately.
By similar calculations as above, we have
\begin{align*}
	\sum_{w \in W_1} d_G(w,A\setminus T') & \ge  \left( 2(a-t+r) + 1 \right) |W_1|,\\
		\sum_{w \in W_2} d_G(w,A \setminus T') & \ge (2(a - t+r) -1 ) |W_2|,\\
		\sum_{w \in W_{>2}} d_G(w,A \setminus T') & \ge (2(a - t)+r -1 ) |W_{>2}|.
\end{align*}
Furthermore,
\begin{align*}
	\sum_{w \in W^*} d_G(w,T') & = e_G(W^*,T')
= \sum_{V \in \mathcal{T} \setminus (\mathcal{R} \cup \{W\}) } e_G(W^*,V) \ge \sum_{V \in \mathcal{T} \setminus (\mathcal{R} \cup \{W\}) } 2(|V| -  \Delta^2) \\
 & \overset{\mathclap{\text{Lemma~\ref{lemma:common-properties}\ref{a}}}}{\ge}
     (t-r-1) \cdot 2(|W^*|-1 - \Delta^2  )
	 \ge 2(t-r-1)|W^*| - 2 \Delta^3,
\end{align*}
where $2( t-r-1 )( \Delta^2 +1 ) \le 2 ( k -1) ( \Delta^2 +1 )  \le 2 (\Delta - 1) ( \Delta^2 +1 ) \le 2 \Delta^3 $.
In summary, we have
\begin{align*}
\sum_{w \in W^*} d_G(w,A) \ge (2a-1) |W^*| - 2 |W_2| -  (r+2)|W_{>2}|-2 \Delta^3.
\end{align*}
By~\eqref{eqn:d=1:u1} taking~$S=W^*$, we have
\begin{align*}
		e_G(B,W^*)  & \le (2b-1) |W^*|  +  2|W_2| + (r+2)|W_{>2}| +2 \Delta^3\\
		& \le (2b-1) |W|  +  2|W_2| + (r+2)|W_{>2}| + 2 \Delta^3\\
		& = (2b+1) |W| + \left( 2|W_2| + (r+2)|W_{>2}|- \frac{4}{3}|W| \right) + \left(2\Delta^3 - \frac23 |W| \right)\\
		& \le  (2b+1) |W| + \frac23 |W_2| + (r+1)|W_{>2}|,
\end{align*}
where the last inequality holds as $|W| \ge \lfloor n/k \rfloor \ge 3\Delta^3$.
Thus, \ref{itm:d=1:e(B,W3)} holds.
\end{proofclaim}

By~\eqref{eqn:B-vertex-degree-in-A}, every vertex in~$B$ has at least two neighbours in~$W^*$.
We now identify those who have at most $3$ neighbours in~$W^*$.
\begin{align*}
	B_1 & =\{x\in B\colon d_G(x,W^*) =2\},  &
	B_2 & =\{x\in B \colon d_G(x,W^*) =3\},\\
	B_{2,1} & =\{x\in B_2 \colon d_G(x,W_1) \ge 2\}, &
	B_{2,2} & = B_2 \setminus B_{2,1}.
\end{align*}

In the next claim, we bound the number of vertices in $B_1$ and~$B_2$.
\begin{claim} \label{clm:d=1:B}
The following inequalities hold.
\begin{enumerate}[label={\rm(\alph*)},start=1]
	\item $2|B_1| \le e_G(B,W_1)$. \label{itm:d=1:B1}
	\item $2|B_1| + |B_2| \le e_G(B,W_1) +\frac13 e_G(B, W_2)$.\label{itm:d=1:B2}
\end{enumerate}
\end{claim}

\begin{proofclaim}
Observe that the following two inequalities imply the claim.
\begin{align}
		2|B_1| + 2 |B_{2,1}| &\le e_G(B,W_1), \label{eqn:d=1:B_1}\\
	2|B_1| + 2 |B_{2,1}| + 3|B_{2,2}| & \le e_G(B,W_1)+e_G(B,W_2). \label{eqn:d=1:B_2}
\end{align}
By Lemma~\ref{lemma:common-properties}\ref{f},  $N_G(B_1, W^*) \subseteq W_1$.
Together with the definition of~$B_{2,1}$, we have
\begin{align*}
	2|B_1| + 2 |B_{2,1}| \le \sum_{y \in B_1 \cup B_{2,1} } d_{G}(y,W_1) \le e_G(B,W_1)
\end{align*}
implying~\eqref{eqn:d=1:B_1}.

Suppose that $N_G(B_{2,2}, W^*) \subseteq W_1 \cup W_2$.
Thus,
\begin{align*}
	2|B_1| + 2 |B_{2,1}| + 3|B_{2,2}| & \le \sum_{y \in B_1 \cup B_{2,1} } d_{G}(y,W_1) +  \sum_{y \in B_{2,2} } d_{G}(y,W_1 \cup W_2)
	 \le e_G(B,W_1)+e_G(B,W_2)
\end{align*}
implying~\eqref{eqn:d=1:B_2}.

Therefore, we may assume that there exists $x \in B_{2,2}$ with $N_G(x,W_{>2}) \ne \emptyset$.
By our induction hypothesis, $G[B]-x$ has an equitable $1$-degenerate $b$-colouring~$\psi_B$.
Recall that $G$ does not have an equitable $1$-degenerate $k$-colouring.
Thus to obtain a contradiction (and to prove the claim), it suffices to show that $G[A \cup \{x\}]$ contains a $1$-degenerate $a$-colouring~$\psi_A$ in which the size of colour class~$C_-$ increases by one and the size of each remaining colour class in~$\mathcal{A}$ remains the same.

Since $x \notin B_{2,1}$, there exist distinct vertices $w_1 \in N_G(x,W_2 \cup W_{>2})$ and $w_2 \in N_G(x, W_{>2})$.
Since $w_1 \notin W_1$, there exists $V_1 \in \mathcal{T} \cup \{U_-\}$ such that $d(w_1, V^*_1) \le 1$, so $w_1$ is movable to~$V_1$.
By Lemma~\ref{lemma:common-properties}\ref{d}, there exists a directed $(V_1, C_-)$-path~$\mathcal{P}$ in~$\mathcal{D}-W$.
Note that $U_- \in V(\mathcal{P})$.
Similarly, since $w_2 \in W_{>2}$, there exist $U_1,U_2 \in \mathcal{R}'$ such that $w_2$ is movable to both $U_1$ and~$U_2$.

\noindent\textbf{Case 1: $U_1, U_2 \in V(\mathcal{P})$.}
Let $U_1$ be closer to~$V_1$ along~$\mathcal{P}$, so $V(V_1\mathcal{P}U_1) \cap V(U_2 \mathcal{P}C_-) = \emptyset$ and $U_1 \in \mathcal{R}$.
Since  $U_1 \in \mathcal{R}$, we have $e_G(U_1, W^*) < 2 (|U_1| -  \Delta^2)$.
There exist $\Delta^2+1$ vertices $u \in U_1$ such that $d_G(u, W^*) \le 1 $.
On the other hand,  there are at most $\Delta^2$ vertices that is joined to~$x$ with a path of length at most~$2$ and~$x \notin U_1$.
Thus, we can pick $u_1 \in U_1$ be such that $d(u_1, W^*) \le 1 $ and there is no path of length at most~$2$ between $u_1$ and~$x$, that is,
\begin{align}
	(\{x\} \cup N_G(x) ) \cap N_G(u_1) = \emptyset. \label{eqn:N(x)N(u1)}
\end{align}

We move $w_1$ to~$V_1$, a vertex of~$V_1$ to~$U_1$ along~$V_1\mathcal{P}U_1$, $w_2$ to~$U_2$, a vertex of~$U_2$ to~$C_-$ along~$U_2 \mathcal{P}C_-$ and both $x$ and $u_1$ to~$W$. Denote by~$\psi_A$ the resulting $a$-colouring of~$G[A \cup \{x\}]$.
The size of each colour class of~$\psi_A$ is as desired.

It remains to show that $\psi_A$ is indeed a $1$-degenerate colouring of~$G[A \cup \{x\}]$.
Each colour class~$V'$ in~$\psi_A$ not corresponding to~$W$ is either the same as before or is obtained by adding a movable vertex and deleting at most one vertex, so $G[V']$ is $1$-degenerate.
The colour class in~$\psi_A$ corresponding to~$W$ is $W_0 = W \cup \{x,u_1\} \setminus \{w_1,w_2\}$.
By Fact~\ref{lemma:d-degenerate}\ref{d-degenerate-a}, $G[W_0 \setminus \{u_1\}]$ is $1$-degenerate.
Observe that $(W_0 \setminus \{u_1\})^* \setminus W^*$ is a subset of isolated vertices in $G[W]$ and so $(W_0 \setminus \{u_1\})^* \setminus W^* \subseteq N_G(x)$.
By~\eqref{eqn:N(x)N(u1)}, we have $d_G(u_1, (W_0 \setminus \{u_1\})^*) \le d_G(u_1, W^*) \le 1$.
By Fact~\ref{lemma:d-degenerate}\ref{d-degenerate-a}, $G[W_0 ]$ is $1$-degenerate as required.

\noindent\textbf{Case 2: $U_i \notin V(\mathcal{P})$ for some $i \in [2]$.}
Since $U_- \in V(\mathcal{P})$, $U_i \in \mathcal{R}$ and so $e_G(U_i, W^*) < 2 (|U_i| -  \Delta^2)$.
By a similar argument as above there exists $u \in U_i$ such that $d_G(u, W^*) \le 1 $ and $(\{x\} \cup	N_G(x)) \cap N_G(u) = \emptyset$.
We move $w_1$ to~$V_1$, a vertex~$V_1$ to~$C_-$ along~$\mathcal{P}$, $w_2$ to~$U_i$ and both $x$ and $u$ to~$W$.
Again, by a similar argument, this results in a desired $1$-degenerate $a$-colouring of~$G[A \cup \{x\}]$.
\end{proofclaim}

We now divide into two cases depending on whether $b =1$ or not.

\noindent \textbf{Case A: $b=1$}.
By Lemma~\ref{lemma:common-properties}\ref{h} with $\tau = 2$, we deduce that $ t =1$.
This implies that $\mathcal{T} = \{W\}$, $\mathcal{R} = \emptyset$ and $W_{>2} = \emptyset$.
Note that
\begin{align*}
	e_G(B,W_1)+e_G(B,W_2) & = e_G(B,W^*) \ge 3|B| - |B_1|
	\overset{\mathclap{\text{Claim~\ref{clm:d=1:B}\ref{itm:d=1:B1}}}}{\ge}
	 3|B| -  e_G(B,W_1)/2,\\
	3|B| &\le 3e_G(B,W_1)/2+e_G(B,W_2)
		\overset{\mathclap{\text{Claim~\ref{claim:d=1upper}\ref{itm:d=1:e(B,W1)} and~\ref{itm:d=1:e(B,W2)}}}}{\le} 3|W_1|/2 + 3 |W_2| \le 3 |W|
		 \overset{\mathclap{\text{Lemma~\ref{lemma:common-properties}\ref{b}}}}{<} 3 |B|,
\end{align*}
a contradiction.

\noindent \textbf{Case B: $b\ge 2$}.
First we bound $e_G(B,W^*)$ from below.
We have
\begin{align*}
	e_G(B,W^*) &\ge 4 |B| - (2|B_1| + |B_2|) \overset{\mathclap{\text{Claim~\ref{clm:d=1:B}\ref{itm:d=1:B2}}}}{\ge}
	4 |B| -e_G(B,W_1) -\frac13 e_G(B, W_2) \\
	& \overset{\mathclap{\text{Claim~\ref{claim:d=1upper}\ref{itm:d=1:e(B,W1)} and~\ref{itm:d=1:e(B,W2)}}}}{\ge} 4 |B| - (2b-1)|W_1| - \frac{1}{3}(2b+t-r)|W_2|\\
	& = 4 |B| - (2b-1)(|W^*| - |W_{>2}|) - \frac{1}{3}(t-r+3-4b)|W_2|\\
	& \ge 4 |B| - (2b-1)(|W| - |W_{>2}|) - \frac{1}{3}(t-r+3-4b)|W_2|.
\end{align*}
Together with Claim~\ref{claim:d=1upper}\ref{itm:d=1:e(B,W3)}, we have
\begin{align*}
	\frac{1}{3}(t-r+5-4b)|W_2|+(r+2-2b)|W_{>2}| \ge 4(|B|- b|W|)	\overset{\mathclap{\text{Lemma~\ref{lemma:common-properties}\ref{b}}}}{>}0.
\end{align*}
By considering the coefficients of $|W_2|$ and $|W_{>2}|$, we must have
\begin{align*}
	\max\left\{ t-r+5-4b, r+2-2b \right\} > 0.
\end{align*}
By Lemma~\ref{lemma:common-properties}\ref{h} with $\tau = 2$, we deduce that $ t \le 2b-1$.
However, $t-r+5-4b \le 4-2b \le 0$ as $b \ge 2$ and $r+2 -2b \le  t-1+2-2b\le 0 $, a contradiction.
This completes the proof of the theorem.
\end{proof}

Here we highlight an obstacle when generalising this proof for equitable $d$-degenerate colourings when $d \ge 2$.
This is due to the dramatic change in $W^*$ after adding a movable vertex to it, which is crucial in the proof of Claim~\ref{clm:d=1:B}\ref{itm:d=1:B2}.
We illustrate this with the following example.
Suppose that $d=2$ and $G$ is a path.
Hence $(V(G))^* = \emptyset$.
Let $G'$ be the new graph obtained from~$G$ by adding a new vertex~$x$ and joining $x$ to both endvertices of~$G$.
Observe that $G'$ is a cycle and so $(V(G'))^* = V(G')$.
(On the other hand, when $d =1$, adding a new vertex to~$x$ increases~$|(V(G))^*|$ by at most the degree of~$x$.)
One could mimic the proof of Theorem~\ref{thm:treecolouring} by setting $\mathcal{R}$ to be the set of colour classes~$V \in \mathcal{T}$ such that $e_G(V,W)$ is small (instead of $e_G(V,W^*)$).
We believe that this would yield an improved bound of saying $k \ge \Delta/(d + 1/2) +10$ for large graphs.
However new idea would be needed to prove that Conjecture~\ref{conj:EDC} holds up to an additive constant.

	\bibliographystyle{abbrv}
	\bibliography{eq}

\end{document}